\documentclass[12pt]{l4dc2021}
\usepackage{times}

\makeatletter
\def\set@curr@file#1{\def\@curr@file{#1}} %
\makeatother

\usepackage{enumitem}
\usepackage{mathdots}
\usepackage{hyperref}
\usepackage{mathrsfs}
\usepackage{tikz}
\usepackage{tikz-3dplot}
\usepackage{wrapfig}
\usepackage{adjustbox}
\usepackage{thmtools, thm-restate}

\usepackage{jptext}
\usepackage{jpmath}

\newcommand{\bSigma}{\boldsymbol\Sigma}

\newcommand{\Kopt}[1]{K^\star_{#1}}
\newcommand{\kopt}[1]{k^\star_{#1}}

\newcommand{\breadth}{\theta}

\newcommand{\Envs}{\Phi}
\newcommand{\env}{\phi}

\newcommand{\Refpolicies}{\Pi_{\mathrm{ref}}}

\newcommand{\envword}{task}
\newcommand{\envwords}{tasks}
\newcommand{\multienv}{multi-task}
\newcommand{\Multienv}{Multi-task}
\newcommand{\singleenv}{single-task}

\newcommand{\multidyn}{multi-dynamics}
\newcommand{\Multidyn}{Multi-dynamics}
\newcommand{\DDFproblem}{DDF problem}
\newcommand{\DDFproblems}{DDF problems}

\newcommand{\Aset}{\mathbf{A}}
\newcommand{\Bset}{\mathbf{B}}

\newcommand{\subopt}{\alpha}
\newcommand{\suboptneighb}[2]{\mathscr{N}_{#1}(#2)}
\newcommand{\coveringset}{\cC}
\newcommand{\coveringnum}[2]{N_{#1}(#2)}

\newcommand{\ratiolb}{\tilde r}

\newcommand{\Ric}{\mathrm{Ric}}
\newcommand{\GCC}{{\normalfont \texttt{GCC}}}

\newcommand{\pos}{\mathbf{x}}
\newcommand{\vel}{\mathbf{v}}
\newcommand{\attitude}{\mathbf{r}}
\newcommand{\anglevel}{\boldsymbol{\omega}}
\newcommand{\roll}{\phi}
\newcommand{\pitch}{\theta}
\newcommand{\yaw}{\psi}

\graphicspath{{figures/}{tikz/}}

\usepackage[utf8]{inputenc}
\DeclareUnicodeCharacter{2212}{-}

\newcommand\inputpgf[2]{{
\let\pgfimageWithoutPath\pgfimage
\renewcommand{\pgfimage}[2][]{\pgfimageWithoutPath[##1]{#1/##2}}
\input{#1/#2}
}}

\getenv[\ABRIDGED]{ABRIDGED}
\newcommand\ifextended[2]{
   \ifdefstring{\ABRIDGED}{true}{
      #2
   }{
      #1
   }
}

\ifextended{
}{
   \jmlrvolume{vol 144}
   \jmlryear{2021}
   \jmlrproceedings{PMLR}{Proceedings of Machine Learning Research}
}

\title{%
	Suboptimal coverings for continuous spaces of control tasks
}
\author{%
   \Name{James A. Preiss} \Email{japreiss@usc.edu} \and
   \Name{Gaurav S. Sukhatme} \Email{gaurav@usc.edu} \\
   \addr University of Southern California, Los Angeles, USA
}
\date{\today}

\begin{document}

\maketitle

\begin{abstract}
	We propose the $\alpha$-suboptimal covering number
	to characterize \multienv\ control problems
	where the set of dynamical systems and/or cost functions is infinite,
	analogous to the cardinality of finite \envword\ sets.
	This notion may help quantify
	the function class expressiveness
	needed to represent a good \multienv\ policy,
	which is important for learning-based control methods that use parameterized function approximation.
	We study suboptimal covering numbers for
	linear dynamical systems with quadratic cost (LQR problems)
	and construct a class of \multienv\ LQR problems amenable to analysis.
	For the scalar case, we show logarithmic dependence on the ``breadth'' of the space.
	For the matrix case, we present experiments
	1) measuring the efficiency of a particular constructive cover, and
	2) visualizing the behavior of two candidate systems for the lower bound.%
	\ifextended{%
	}{%
		An extended version is available at \url{jpreiss.github.io/pubs/suboptimal_coverings.pdf}.
	}
\end{abstract}

\section{Introduction}
An advanced control system such as a mobile robot may be required to perform many different tasks.
If the \envword\ set is finite, like selecting between ``map an environment'' and ``deliver a package'',
then its size is naturally quantified by the number of \envwords.
If the \envword\ set is infinite, like delivering packages with arbitrary mass and inertial properties,
then its size is not so easily quantified.
Even if the task space is equipped with a metric or measure,
these structures may be only weakly linked to the diversity of behavior
required for good performance on all \envwords.

Our interest in this issue is motivated by \multienv\ paradigms in learning-based control,
where the policy is selected from a parameterized family of functions
that map state and \envword\ parameters directly to actions.
As the \envword\ space expands from a singleton set,
we expect to need a more expressive class of functions
to represent a good \multienv\ policy.
In this work, we propose the 
\emph{$\alpha$-suboptimal covering number}
to capture this idea. %
For a \envword\ space $\Envs$
and a suboptimality ratio $\alpha > 1$,
we define $\coveringnum{\alpha}{\Envs}$ as
the size of the smallest set of \singleenv\ policies $\coveringset$ %
such that for every $\env \in \Envs$, at least one $\pi \in \coveringset$
has a cost ratio no greater than $\alpha$ relative to the optimal policy for $\env$.
If the policies in $\coveringset$ are parameterized functions,
then $\coveringset$ provides an upper bound on the number of parameters
needed to represent an $\alpha$-suboptimal \multienv\ policy.
In switching-based adaptive control, where $\env$ is unknown,
a smaller $\coveringset$ implies a faster convergence time.

To study suboptimal covering numbers in a concrete setting,
we consider linear dynamical systems
with quadratic cost functions, or LQR problems.
LQR problems are a common setting to analyze learning algorithms because detailed properties are known \citep[for example]{fazel-global-lqr}.
This has led to new inquiries into their fundamental properties \citep{bu-topological-gains}.
We construct a family of well-behaved \multienv\ LQR problems 
where $\Envs$ is controlled by a ``breadth'' parameter $\breadth \in [1, \infty)$,
and for which
$\coveringnum{\alpha}{\Envs_\breadth}$
is finite and increasing in $\breadth$.
For the special case of a scalar LQR problem,
we derive matching logarithmic upper and lower bounds on
$\coveringnum{\alpha}{\Envs_\breadth}$
as a function of $\breadth$.
As an effort towards analogous bounds for the matrix case,
we present empirical results intended to shed light on the problem structure.
For the upper bound, we analyze properties of a logical extension of our scalar cover.
For the lower bound, we visualize suboptimal neighborhoods for two choices of ``extremal''  systems
and find surprising topological behavior for one choice.

This paper is an initial step towards a comprehensive theory.
In addition to a more complete picture of deterministic LQR systems,
ideas of $\alpha$-suboptimal coverings could be applied to a wide range of \multienv\ problems.
We also hope they will lead to insights
about function class expressiveness in learning-based \multienv\ control.

\section{Problem setting}
In this section, we first define suboptimal covering numbers
with respect to an abstract \multienv\ control problem
independent of distinctions
such as continuous vs.\ discrete time and stochastic vs.\ deterministic.
We then instantiate these notions for a particular class of LQR problems.

\paragraph{Notation}
The set of all functions $\cX \mapsto \cY$ is denoted by $\cY^\cX$.
The relation $A \succeq B$ (resp. ${A \succ B}$) denotes that $A - B$ is positive semidefinite (resp.\ definite).
Matrices of zeros and ones, with dimension implied by context, are denoted by $\zero$ and $\one$.
The integers $\{1, \dots, N\}$ are denoted by $[N]$.

\paragraph{\Multienv\ optimal control}
	A \emph{\multienv\ optimal control problem} is defined by
	an arbitrary state space $\cX$,
	action space $\cU$,
	and \envword\ space $\Envs$;
	a class of reference policies $\Refpolicies \subseteq \cU^\cX$,
	and a strictly positive objective function $J : \Envs \times \cU^\cX \mapsto \R_{> 0}$.
	The partial application of $J$ for $\env \in \Envs$
	is denoted by $J_\env : \cU^\cX \mapsto \R$.
	The optimal reference cost for an \envword\ is denoted by
	$J^\star_\env = \min_{\pi \in \Refpolicies} J_\env(\pi)$.

\paragraph{Suboptimal coverings}
	Consider a \multienv\ optimal control problem $(\cX, \cU, \Envs, \Refpolicies)$
	and a suboptimality ratio $\alpha > 1$.
	The \emph{$\subopt$-suboptimal neighborhood}
	of the policy $\pi : \cX \mapsto \cU$ is
	\(
		\suboptneighb{\subopt}{\pi} = \left\{
			\env \in \Envs :
			\nofrac{J_\env(\pi)}{J^\star_\env}
			\leq \subopt
		\right\}.
	\)
	The set $\coveringset \subseteq \cU^\cX$
	is an \emph{$\subopt$-suboptimal cover} of $\ \Envs$
	if
	\(
		\ \bigcup_{\pi \in \coveringset} \suboptneighb{\subopt}{\pi}
		= \Envs.
	\)
	The \emph{$\subopt$-suboptimal covering number} of $\Envs$,
	denoted $\coveringnum{\subopt}{\Envs}$,
	is the size of the smallest finite $\alpha$-suboptimal cover of $\Envs$ if one exists,
	or $\infty$ otherwise.

\paragraph{Standard LQR problem}
A continuous-time, deterministic, infinite-horizon, time-invariant LQR problem with full-state feedback is defined by 
state space $\cX = \R^n$,
action space $\cU = \R^m$,
linear dynamics
\(
	\dot x = Ax + Bu,
\)
where $A \in \R^{n \times n},\ B \in \R^{n \times m}$, and quadratic cost 
\[
	J(\pi) = \E_{x(0) \sim \Normal(\zero, I)}
		\int_0^\infty \big(
			x^\top Q x + u^\top R u
		\big)\, \d t,
\]
where $Q \succeq \zero$ and $R \succ \zero$ are cost matrices of appropriate dimensions
and $\Normal(\zero, I)$ is the unit Gaussian distribution.
For the purposes of this paper, the pair $(A, B)$ is \emph{controllable}
if $J(\pi) < \infty$ for some policy $\pi$.
If $(A, B)$ is controllable,
then the optimal policy %
is the linear
$u = \Kopt{} x$,
where $\Kopt{} \in \R^{m \times n}$
can be computed
by finding the unique maximal positive semidefinite solution $P$
of the algebraic Riccati equation
\(
	A^\top P + P A - P B R^{-1} B^\top P + Q = \zero
\)
(henceforth called the \emph{maximal solution})
and letting $\Kopt{} = -R^{-1} B^\top P$
\citep{kalman-contributions}.
Additionally, $J(\Kopt{}) = \trace{P}$.
An arbitrary controller $K \in \R^{m \times n}$ is \emph{stabilizing} if
$J(K) < \infty$,
in which case $J(K)$ satisfies
\begin{equation}
\label{eq:cost-suboptimal}
	J(K) = \trace{(Q + K^\top R K) W}\!, \
\text{where} \ \ 
	W = \int_0^\infty {e^{t(A + BK)}}^\top {e^{t(A + BK)}} dt.
\end{equation}
$W$ can be computed by solving the Lyapunov equation
\(
	(A + BK)^\top W + W (A + BK)  + I = \zero.
\)

\paragraph{\Multidyn\ LQR}
\label{sec:multilqr}
A fully general formulation of \multienv\ LQR
would allow variations in each of $(A, B, Q, R)$,
but this creates redundancy.
Any LQR problem
where $Q \succ 0$
is equivalent via change of coordinates
to another LQR problem where $Q = I$ and $R = I$.
To reduce redundancy,
we consider only \emph{\multidyn} LQR problems where $Q = I_{n \times n}$ and $R = I_{m \times m}$
in this work.
The reference policy class is linear: $\Refpolicies = \R^{m \times n}$.

A \multidyn\ LQR problem can be defined by
$\Envs = \Aset \times \Bset$
for some sets $\Aset \subseteq \R^{n \times n}$ and ${\Bset \subseteq \R^{n \times m}}$,
but it is not obvious how to design $\Aset$ and $\Bset$.
To support an asymptotic analysis of
$\coveringnum{\subopt}{\Envs}$,
the \envword\ space $\Envs$ should have a real-valued ``breadth'' parameter $\breadth$
that sweeps from a single \envword\ to sets with arbitrarily large, but finite, covering numbers.
Matrix norm balls are a popular representation of dynamics uncertainty in the robust control literature,
but they can easily contain uncontrollable pairs,
and removing the uncontrollable pairs
can lead to an infinite covering number.
For example, in the scalar problem $\Aset = \{a\},\ \Bset = [-\breadth, 0) \cup (0, \breadth]$, where $a > 0$,
it can be shown that
no $\alpha$-suboptimal cover is finite.

These properties are worrying, but the example $\Bset$ is pathological.
The zero crossing is analogous to reversing the direction of force applied by an actuator in a physical system.
A more relevant multi-dynamics problem is variations
in mass or actuator strength,
whose signs are fixed.
We formalize this idea with the following definition.
\paragraph{Decomposed dynamics form}
\label{def:B-SVD}
	Fix $A \in \R^{n \times n}$
	and a \emph{breadth} parameter $\breadth \geq 1$.
	Let
		$\Bset = \{ U \Sigma V^\top : \Sigma \in \bSigma \}$,
		where
		$\bSigma = \{ \diag(\sigma) : \sigma \in [\recip \breadth, 1]^d \}$.
	The matrices $U \in \R^{n \times d}$
	and
	$V \in \R^{m \times d}$ each have rank $d$, where
	$0 < d \leq \min \{n, m\}$.
	The tuple $(A, U, V, \breadth)$
	fully defines a \emph{\multienv\ LQR problem in decomposed dynamics form},
	or \emph{\DDFproblem} for brevity.

\paragraph{}
The continuity of the LQR cost \eqref{eq:cost-suboptimal} with respect to $B$
and the compactness of $\Envs$ for any $\breadth$ imply that
$\coveringnum{\alpha}{\Envs_\breadth}$ is always finite.
Variations in $A$ are redundant in the scalar case where we focus our theoretical work in this paper.
The definition can be extended to include them in future work.

\paragraph{Linearized quadrotor example}

\label{sec:quadrotor}
As an example of a realistic \DDFproblem,
we consider the quadrotor helicopter illustrated in \Cref{fig:quadrotor}.
Near the hover state, its full nonlinear dynamics are well approximated by a linearization.
The state is given by
$
	x = (\pos, \vel, \attitude, \anglevel),
$
where $\pos \in \R^3$ is position,
$\vel \in \R^3$ is linear velocity,
$\attitude \in \R^3$ is attitude Euler angles,
and $\anglevel \in \R^3$ is angular velocity.
The inputs
$u \in \R^4_{\geq 0}$
are the squared angular velocities of the propellers.

Many factors influence the response to inputs,
including geometry, mass, moments of inertia, motor properties, and propeller aerodynamics.
These can be combined and partially nondimensionalized
into four control authority parameters to form $\env \in \Envs$.
The hover state occurs at $x = \zero,\ u \propto \one$,
where the constant input counteracts gravity.
The linearized dynamics are given by
\newcommand{\Gmatrix}{\begin{bsmallmatrix*}[r] 0 & g & 0 \\ \hspace{-1mm}-g & 0 & 0 \\ 0 & 0 & 0 \\ \end{bsmallmatrix*}}
\begin{equation*}
\label{eq:quad-linearized}
	\dot x
	=
	\underbrace{
		\begin{bsmallmatrix}
			\zero &     I & \zero    & \zero \\
			\zero & \zero & G & \zero \\
			\zero & \zero & \zero    &     I \\
			\zero & \zero & \zero    & \zero \\
		\end{bsmallmatrix}
	}_A
	x
	+
	\underbrace{
		\begin{bsmallmatrix}
			\zero    & \zero \\
			\hat e_z & \zero \\
			\zero    & \zero \\
			\zero    &     I \\
		\end{bsmallmatrix}
	}_U
	\underbrace{
		\begin{bsmallmatrix}
			\sigma_z & & & \\
			& \sigma_\roll & & \\
			& & \sigma_\pitch & \\
			& & & \sigma_\yaw \\
		\end{bsmallmatrix}
	}_\Sigma
	\underbrace{
		\begin{bsmallmatrix*}[r]
			 1 &  1 &  1 &  1 \\
			 1 & -1 & -1 &  1 \\
			-1 & -1 &  1 &  1 \\
			 1 & -1 &  1 & -1 \\
		\end{bsmallmatrix*}
	}_{V^\top}
	u,
	\quad
	G = \Gmatrix,
\end{equation*}
where $g$ is the gravitational constant
and $\hat e_z = [0\ 0\ 1]^\top$.
The parameters
$(\sigma_z,\ \sigma_\roll,\ \sigma_\pitch,\ \sigma_\yaw)$
denote
the thrust, roll, pitch, and yaw authority constants respectively.
Since we use the convention $\sigma \in [\recip \theta, 1]$,
the maximum value of each constant can be varied
by scaling the columns of $U$.

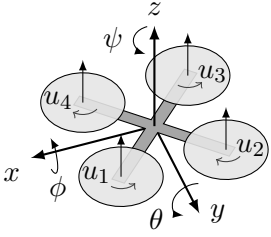
\begin{figure}[tpb]
	\begin{minipage}{0.28\textwidth}
		\tikzstyle{axis} = [-latex, thick]
\tikzstyle{thrust} = [-latex]
\tikzstyle{rot} = [-latex]
\tikzstyle{propdisc} = [fill=gray!20, fill opacity=0.90]
\tikzstyle{propmotion} = [gray!80]
\tikzstyle{propcw} = [->,black!75]
\tikzstyle{propccw} = [<-,black!75]

\tdplotsetmaincoords{45}{160}
\begin{tikzpicture}[scale=0.70,tdplot_main_coords]

  \pgfmathsetmacro{\xyaxlen}{2.5}
  \pgfmathsetmacro{\zaxlen}{2.7}
  \draw[axis] (0,0,0) -- (\xyaxlen,0,0) node[anchor=north east]{$x$};
  \draw[axis] (0,0,0) -- (0,\xyaxlen,0) node[anchor=west]{$y$};

  \tdplotsetrotatedcoords{0}{0}{45}
  \begin{scope}[tdplot_rotated_coords]
    \pgfmathsetmacro{\thick}{0.13}
    \pgfmathsetmacro{\shaftctr}{1.5}
    \pgfmathsetmacro{\len}{\shaftctr+\thick}
    \pgfmathsetmacro{\z}{0.0}
    \draw[fill=gray!70,tdplot_rotated_coords]
      (-\len,\thick,\z) --
      (-\thick,\thick,\z) --
      (-\thick,\len,\z) --
      (\thick,\len,\z) --
      (\thick,\thick,\z) --
      (\len,\thick,\z) --
      (\len,-\thick,\z) --
      (\thick,-\thick,\z) --
      (\thick,-\len,\z) --
      (-\thick,-\len,\z) --
      (-\thick,-\thick,\z) --
      (-\len,-\thick,\z) --
      cycle;

    \newcommand{\prop}[3]{
      \pgfmathsetmacro{\propradius}{0.8}
      \draw[propdisc] (#1,#2,0) circle (\propradius);
      \tdplotdefinepoints(#1,#2,0)(#1-0.15,#2+0.2,0)(#1+0.3,#2+0.2,0)
      \tdplotdrawpolytopearc[#3]{0.4*\propradius}{}{}
    }
    \prop{\shaftctr}{0}{propccw}
    \prop{0}{\shaftctr}{propcw}
    \prop{-\shaftctr}{0}{propccw}
    \prop{0}{-\shaftctr}{propcw}

    \pgfmathsetmacro{\thrusttext}{0.05}
    \pgfmathsetmacro{\thrustlen}{1.15}
    \draw[thrust] ( \shaftctr,0,0) -- ++(0,0,\thrustlen) node[pos=\thrusttext, left] {$u_1$};
    \draw[thrust] (0, \shaftctr,0) -- ++(0,0,\thrustlen) node[pos=\thrusttext,right] {$u_2$};
    \draw[thrust] (-\shaftctr,0,0) -- ++(0,0,\thrustlen) node[pos=\thrusttext,right] {$u_3$};
    \draw[thrust] (0,-\shaftctr,0) -- ++(0,0,\thrustlen) node[pos=\thrusttext, left] {$u_4$};
  \end{scope}

  \draw[axis] (0,0,0) -- (0,0,\zaxlen) node[anchor=south]{$z$};

  \pgfmathsetmacro{\arcinset}{0.4}
  \pgfmathsetmacro{\arcradius}{0.4}
  \tdplotdefinepoints(\xyaxlen-\arcinset,0,0)(\xyaxlen-\arcinset,0.1,-1)(\xyaxlen-\arcinset,0,1)
  \tdplotdrawpolytopearc[rot]{\arcradius}{anchor=north}{$\phi$}
  \tdplotdefinepoints(0,\xyaxlen-\arcinset,0)(-0.1,\xyaxlen-\arcinset,0.3)(0.1,\xyaxlen-\arcinset,-.2)
  \tdplotdrawpolytopearc[rot]{\arcradius}{anchor=north east}{$\theta$}
  \tdplotdefinepoints(0,0,\zaxlen-\arcinset)(0,-1,\zaxlen-\arcinset)(1,1.0,\zaxlen-\arcinset)
  \tdplotdrawpolytopearc[rot]{\arcradius}{anchor=east}{$\psi$}

\end{tikzpicture}
	\end{minipage}
	\hfill
	\begin{minipage}{0.68\textwidth}
		\caption{%
			Quadrotor helicopter with position states $x, y, z$,
			attitude states $\roll, \pitch, \yaw$,
			and propeller speed inputs $u_1, u_2, u_3, u_4$.
			The linearized dynamics at hover,
			subject to variations in mass, geometry, etc.,
			can be expressed in decomposed dynamics form---%
			see \Cref{sec:quadrotor}.
		}%
		\label{fig:quadrotor}
	\end{minipage}
\end{figure}

\section{Theoretical results}
\label{sec:results-scalar}

In this section we show logarithmic upper and lower bounds on
the growth of $\coveringnum{\alpha}{\Envs_\breadth}$ in $\breadth$
for scalar \DDFproblems.
We present several intermediate results in matrix form because they are needed for our empirical results later.
We begin with a key lemma in the framework of
\emph{guaranteed cost control} (GCC)
from \citet{petersen-guaranteed-cost},
simplified for our use case.

\begin{lemma}[GCC synthesis, \citet{petersen-guaranteed-cost}]
\label{lem:petersen-gcc}
	Given the \multienv\ LQR problem
		$\Aset = \{ A \}$,
		$\Bset = \{ B_1 \Delta + B_2 : \norm{\Delta} \leq 1 \}$,
	where $B_1, B_2 \in \R^{m \times p}$ are arbitrary for arbitrary $p$,
	and the state cost matrix is $Q \succ \zero$,
	if there exists $\tau > 0$ such that $P \succ \zero$ solves the Riccati equation
	\begin{equation}
	\label{eq:petersen-riccati}
		A^\top P + P A
		+ P \left(
			\textstyle \recip \tau B_1 B_1^\top - \recip{1+\tau} B_2 B_2^\top
		\right) P
		+ Q = \zero,
	\end{equation}
	then the controller
	\(
		K = -\recip{1 + \tau} B_2^\top P
	\)
	has cost $J_B(K) \leq \trace{P}$
	for all $B \in \Bset$.
	Also, $\trace{P}$ is a convex function of $\tau$.
\end{lemma}

We use the notation
\(
	P, \tau, K = \GCC(A, B_1, B_2, Q)
\)
to indicate that $P, \tau$ solve
\eqref{eq:petersen-riccati}
and $K$ is the corresponding controller.
It is straightforward to show that any \DDFproblem\
can be expressed in the form required by \lemmaref{lem:petersen-gcc}
with additional constraints on $\Delta$.

In the original presentation, \citet{petersen-guaranteed-cost} treat $B_1$ as given,
so they accept that
\eqref{eq:petersen-riccati} may have no solution
(for example, when $A = 2I, B_1 = I, B_2 = \zero$).
Our application requires constructing values of $B_1, B_2$ that guarantee a solution,
motivating the following lemma.
\defcitealias{lancaster-AREs}{Lan95}
We abbreviate the reference text \citet{lancaster-AREs} as \citetalias{lancaster-AREs}.

\begin{lemma}[existence of $\alpha$-suboptimal GCC]
\label{lem:petersen-existence}
	For the \DDFproblem\ $(A, U, V, \theta)$, if $B \in \Bset$ and $\alpha > 1$,
	then there exists $B_1 \neq \zero \in \R^{m \times n}$
	such that the GCC Riccati equation \eqref{eq:petersen-riccati}
	with $B_2 = B$
	has a solution $(P, \tau)$
	satisfying
	$\trace{P} \leq \alpha J^\star_{B}$.
\end{lemma}
\begin{proof}
	For this proof, it will be more convenient to write the algebraic Riccati equation as
	\begin{equation}
		\label{eq:existence-ric}
		A^\top P + P A - P D P + Q = \zero,
	\end{equation}
	where $D \succeq \zero$.
	Let $\cD = \{ D \succeq \zero: (A, D)\ \text{ is controllable} \}$.
	Controllability of $(A, B)$
	implies that
	$BB^\top \in \cD\ $
	\citepalias[Corollary 4.1.3]{lancaster-AREs}.
	Let $\Ric_+$ denote the map from $\cD$ to the
	maximal solution of \eqref{eq:existence-ric},
	which is continuous %
	\citepalias[Theorem 11.2.1]{lancaster-AREs},
	and let $\cD_\alpha = \{D \in \cD : \trace{\Ric_+(D)} < \alpha J^\star_B \}$.
	The set $\cD_\alpha$ is open in $\cD$ by continuity
	and is nonempty because it contains $BB^\top$.
	Now define $B_1(\tau) = \tau B$ for $\tau \in (0, \half)$.
	The equivalent of $D$ in the GCC Riccati equation \eqref{eq:petersen-riccati} becomes
	\[
		\textstyle D(\tau) = -\recip \tau B_1(\tau) B_1(\tau)^\top + \recip{1 + \tau} B_2 B_2^\top
		= \frac{1 - \tau - \tau^2}{1 + \tau} B B^\top.
	\]
	As a positive multiple of $B B^\top$, we know $D(\tau) \in \cD$,
	and because $\lim_{\tau \to 0} D(\tau) = B B^\top$,
	the set of $\tau$ for which $D(\tau) \in \cD_\alpha$ is nonempty.
	Any such $\tau$ and $B_1(\tau)$ provide a solution.
\end{proof}

\noindent Finally, the following comparison result
will be useful in several places.

\begin{lemma}[\citetalias{lancaster-AREs}, Corollary 9.1.6]
\label{lem:lancaster-ARE-domination}
\newcommand{\alt}[1]{\tilde #1}
	Given two algebraic Riccati equations
	\[
		A^\top P + PA - P B B^\top P + Q = \zero
		\quad \text{ and } \quad 
		\alt{A}^\top P + P\alt{A} - P \alt{B}\alt{B}^\top P + \alt{Q} = \zero,
	\]
	with maximal solutions $P$ and $\alt{P}$,
	let
		$X = \bsmallmat{Q & A^\top \\ A & -BB^\top}$
		and
		$\alt{X} = \bsmallmat{\alt{Q} & \alt{A}^\top \\ \alt{A} & -\alt{B}\alt{B}^\top}$.
	If $X \!\succeq \!\alt{X}$, then $P \!\succeq \!\alt{P}$.
\end{lemma}

\subsection{Scalar upper bound}

We are now ready to bound the covering number for scalar systems.
The first lemma bounding $J^\star_{a,b}$ will be useful for the lower bound also.
We then construct a cover inductively.

\begin{lemma}
\label{lem:scalar-cost-ub}
	In a scalar LQR problem,
	if $a > 0$ and $0 < b \leq 1$,
	then the optimal scalar LQR cost satisfies the bounds
	\(
		\nofrac{2a}{b^2}
		<
		J^\star_{a,b}
		<
		\nofrac{(2a + 1)}{b^2}.
	\)
\end{lemma}
\begin{proof}
	The lower bound is visible from
	the closed-form solution for the scalar Riccati equation, which is
	\(
		J^\star_{a,b} = \frac{a + \sqrt{a^2 + b^2}}{b^2}.
	\)
	The upper bound is obtained by substituting $a^2 + b^2 \leq (a + 1)^2$.
\end{proof}

\begin{lemma}
\label{lem:scalar-cover-recursion}
	If $p, \tau, k = \GCC(a, b_1, b_2, q)$,
	then for any $\beta \in (0, 1)$,
	there exists $k' \in \R$
	such that
	$\displaystyle p', \tau, k' = \GCC \left( a, \beta b_1, \beta b_2, \beta^{-2} q \right)$,
	where $\ p' = \beta^{-2} p$.%
\end{lemma}
\begin{proof}
	In the scalar system,
	the GCC matrix Riccati equation \eqref{eq:petersen-riccati}
	reduces to the quadratic equation
	\begin{equation}
	\label{eq:scalar-ric}
		\left(
			\textstyle \recip \tau b_1^2
			- \recip{1 + \tau} b_2^2
		\right)
		p^2
		+ 2a p
		+ q
		= 0.
	\end{equation}
	Substituting $p' = \beta^{-2} p$ into \eqref{eq:scalar-ric}
	and multiplying by $\beta^{-2}$ yields
	a new instance of~\eqref{eq:scalar-ric}
	with the parameters
		$b_1' = \beta b_1$, \ 
		$b_2' = \beta b_2$, \ 
		$q' = \beta^{-2} q$,
	for which $p'$ is a solution
	with $\tau$ unchanged.
\end{proof}

\begin{theorem}
\label{thm:covering-scalar}
	For the scalar \DDFproblem\ defined by
		$\Aset = \{ a \}$,
		where $a > 0$,
		and 
		$\Bset = \left[\textstyle \recip \breadth, 1 \right]$,
	if $\alpha \geq \frac{2a+1}{2a}$,
	then $\coveringnum{\alpha}{\Bset} = O(\log \breadth)$.
\end{theorem}
\begin{proof}
	We construct a cover from the upper end of $\Bset$.
	By \lemmaref{lem:scalar-cost-ub},
	the condition ${\alpha \geq \frac{2a+1}{2a}}$ implies that
	\(
		J^\star_{b=1} < \alpha 2 a < \alpha J^\star_{b=1}.
	\)
	Therefore, by
	\lemmaref{lem:petersen-gcc,lem:petersen-existence},
	there exists $\beta \in (0, 1)$
	and
	$p, \tau, k$
	such that
	$p, \tau, k = \GCC(a, (1-\beta)/2, (1+\beta)/2, 1)$
	and
	$p \leq \alpha 2a$.

	Proceeding inductively,
	suppose that for $N \geq 1$,
	we have covered $[\beta^N, 1]$
	by the intervals $\Bset_n = [\beta^{n+1}, \beta^{n}]$
	for $n \in \{0, \dots, N-1\}$,
	and each $\Bset_n$ has a controller $k_n$ such that 
	\[
		\beta^{-2n}p, \tau, k_n = \GCC \left(
			a,
			\nofrac{(\beta^{n} - \beta^{n+1})}{2},
			\nofrac{(\beta^{n} + \beta^{n+1})}{2},
			\beta^{-2n}
		\right).
	\]
	Then the existence of the desired $\Bset_N, k_N$
	follows immediately from
	\lemmaref{lem:scalar-cover-recursion}.

	By \lemmaref{lem:lancaster-ARE-domination},
	for each $\Bset_n$
	the GCC state cost $q_n = \beta^{-2n} \geq 1$
	is an upper bound on the cost if we replace $q_n$ with $1$
	to match the \DDFproblem.
	Therefore,
	for each interval $\Bset_n$,
	for all $b \in \Bset_n$,
	\[
		\alpha J^\star_{b}
		\geq
		\alpha J^\star_{\beta^n}
		>
		\beta^{-2n} \alpha 2 a 
		\geq
		\beta^{-2n} p
		\geq
		J_b(k_n),
	\]
	where first inequality is due to \lemmaref{lem:lancaster-ARE-domination},
	the second is due to \lemmaref{lem:scalar-cost-ub},
	the third is by construction of $p$,
	and last is due to the GCC guarantee of $k_n$.
	Hence, $\Bset_n \subseteq \suboptneighb{\alpha}{k_n}$.
	We cover the full $\Bset$ when
	$\beta^{N} \leq \recip \theta$,
	which is satisfied by $N \geq - \log \theta / \log \beta$.
\end{proof} 

\subsection{Scalar lower bound}

For the matching lower bound,
we begin by
deriving a simplified overestimate of $\suboptneighb{\alpha}{k}$.
We then show that the true $\suboptneighb{\alpha}{k}$
is still a closed interval moving monotonically with $k$.
Finally, we argue that
the gaps between consecutive elements of a cover grow at most geometrically,
while the range of $k$ values in a cover must grow linearly with $\breadth$.

\begin{lemma}
\label{lem:scalar-neighborhood-optimistic}
	For a scalar \DDFproblem\ with $a \geq 1,\ \Bset = [\recip \breadth, 1]$,
	for any $k < 0$,
	if $\alpha \geq 3/2$, then
	$\suboptneighb{\alpha}{k} \subseteq \recip{|k|} [c_1 \!- \!c_2,\ c_1\! + \!c_2]$,
	where $c_1$ and $c_2$ are constants depending on $\alpha$ and $a$.
\end{lemma}
\begin{proof}
	Beginning with the closed-form solution for $J_b(k)$,
	which can be derived from \eqref{eq:cost-suboptimal},
	we define
	\begin{equation}
	\label{eq:Jbk}
		J_b(k) = \frac{1 + k^2}{-2(a + bk)}
		\geq \frac{k^2}{-2(a + bk)}
		\defeq \underline{J_b}(k)
		.
	\end{equation}
	By \lemmaref{lem:scalar-cost-ub}, we have
	\(
		J^\star_b < \nofrac{3a}{b^2} \defeq \overline{J^\star_b}
	\),
	so $\ratiolb = \nofrac{\underline{J_b}(k)}{\,\overline{J^\star_b}}$
	is a lower bound on the suboptimality of $k$.
	Computing $\partial^2 \ratiolb / \partial b^2$
	shows that $\ratiolb$ is strictly convex in $b$ on the domain $a + bk < 0$,
	so the $\alpha$-sublevel set of $\ratiolb$ is the closed interval
	with boundaries where $\ratiolb = \alpha$.
	This equation is quadratic in $b$ with the solutions
	\(
		b = -a(3 \alpha \pm \sqrt{9 \alpha^2 - 6 \alpha})/k
	\).
	The resulting interval contains $\suboptneighb{\alpha}{k}$.
\end{proof}

\begin{restatable}{lemma2}{quasi}
\label{lem:subopt-convex}
	For a scalar \DDFproblem, if $\alpha > 1$
	and $k < -1$,
	then $\suboptneighb{\alpha}{k}$ is either empty or a closed interval
	$[b_1, b_2]$,
	with $b_1$ and $b_2$ positive and nondecreasing in $k$.
\end{restatable}
\begin{proof}
The result follows from quasiconvexity of
both the suboptimality ratio $J_b(k) / J^\star_b$
and the cost $J_b(k)$.
Showing these requires some tedious calculations.
\ifextended{
	For details, see \Cref{sec:appendix-theory}.
}{
	For details, see the extended version of this paper.
}
\end{proof}

\begin{theorem}
\label{thm:scalar-lb}
	For a scalar \DDFproblem\ with
	$a = 1,\ \Bset = [\recip \breadth, 1]$,
	if $\alpha \geq \frac 3 2$, %
	then
	$\coveringnum{\alpha}{\Bset} = \Omega(\log \breadth)$.
\end{theorem}
\begin{proof}
	From the closed-form solution
	\(
		k^\star_{a,b} = -(a + \sqrt{a^2 + b^2})/b
	\),
	we observe that $k^\star_b < -1$ for all $b \in \Bset$.
	This, along with the quasiconvexity of $J_b(k)$ in $k$,
	implies that there exists a minimal $\alpha$-suboptimal cover $\coveringset$ for which all $k_i < -1$.
	Suppose $\coveringset = k_1, \dots, k_N$ is such a cover, ordered such that $k_i < k_{i+1}$.
	Then
	by \lemmaref{lem:subopt-convex},
	$\suboptneighb{\alpha}{k_i}$
	and
	$\suboptneighb{\alpha}{k_{i+1}}$
	must intersect,
	so their overestimates according to
	\lemmaref{lem:scalar-neighborhood-optimistic}
	certainly intersect,
	therefore satisfying
	\[
		\frac{c_1 + c_2}{-k_{i+1}}
		\geq
		\frac{c_1 - c_2}{-k_i}
		\implies
		\frac{k_{i+1}}{k_i} \leq \frac{c_1 + c_2}{c_1 - c_2}
		\implies
		\frac{k_N}{k_1} \leq \left(\frac{c_1 + c_2}{c_1 - c_2}\right)^{N-1}.
	\]
	By \lemmaref{lem:scalar-neighborhood-optimistic},
	to cover $b = 1$
	controller $k_1$ must satisfy
	$k_1 \geq -(c_1 + c_2)$,
	and to cover $b = \recip \breadth$,
	controller $k_N$ must satisfy
	$k_N \leq -\breadth(c_1 - c_2)$.
	Along with the previous result, this implies
	\[
		\left(\frac{c_1 + c_2}{c_1 - c_2}\right)^{N-1}
		\geq
		\breadth
		\frac{c_1 - c_2}{c_1 + c_2}
		\implies
		N \geq \frac{\log \breadth}{\log \frac{c_1 + c_2}{c_1 - c_2}}.
	\]
	Recalling that $c_1$ and $c_2$ only depend on $a$ and $\alpha$,
	the $\Omega(\log \breadth)$ dependence on $\breadth$ is established.
\end{proof}

\paragraph{Remarks}
\begin{itemize}
	\item For the upper bound, it may be possible to compute or bound $\beta$ in the scalar case
		as a function of $a$ and $\alpha$,
		but the analogous result will likely be much more complicated in the matrix case.
	\item \theoremref{thm:covering-scalar} imposes a lower bound on $\alpha$ greater than $1$.
		We believe this is a mild condition in practice:
		if the application demands a suboptimality ratio very close to 1,
		then the size of the suboptimal cover is likely to become impractical for storage.
		However, further theoretical results building upon suboptimal coverings may require eliminating the bound.
\end{itemize}

\section{Empirical results}
\label{sec:grid}

For matrix \DDFproblems,
we present empirical results as a first step towards covering number bounds. %
We begin by testing a cover construction. %
If the construction fails to achieve a conjectured upper bound in a numerical experiment,
then either the conjecture is false,
or the construction is not efficient.
A natural idea is to extend
the geometrically spaced sequence of $b$ values from
\lemmaref{lem:scalar-cost-ub}
to multiple dimensions.
We now make this notion, illustrated in \Cref{fig:geomgrid}, precise.

\begin{definition}[Geometric grid partition]
\label{def:geomgrid}
	Given a \DDFproblem\ with $\bSigma = [\recip \breadth, 1]^d$,
	and a grid pitch $k \in \N_+$,
	select
	$s_1, \dots, s_{k+1}$
	such that
	$s_1 = \recip \breadth$,
	$s_{k + 1} = 1$,
	and $\frac{s_{i+1}}{s_i} > 0$
	is constant.
	For each $j \in [k]^d$, define the grid cell
	\(
		\bSigma(j) = \textstyle \prod_{i=1}^d [s_{j(i)}, s_{j(i) + 1}],
	\)
	where $j(i)$ is the $i\th$ component of $j$.
	The cells satisfy
	\(
		\bSigma = \textstyle \bigcup_{j \in [k]^d} \bSigma(j),
	\)
	thus forming an partition (up to boundaries) of $\bSigma$ into $k^d$ cells.
\end{definition}

\begin{figure}
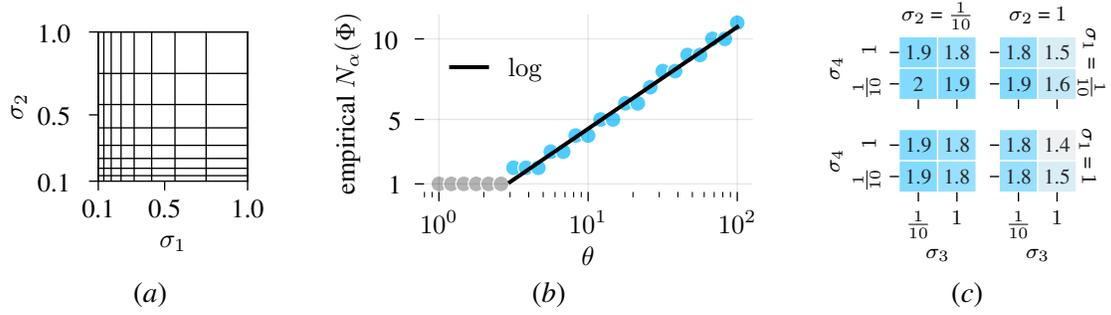
%
	\subfigure{
		\hspace{-5mm}
		\raisebox{0mm}{
			\input{figures/geomgrid.pgf}
		}
		\label{fig:geomgrid}
	}%
	~\hfill
	\subfigure{%
		\raisebox{-2mm}{
			\input{figures/covering_quadrotor.pgf}
		}
		\label{fig:quadrotor_covering}
	}%
	~\hfill
	\subfigure{%
		\raisebox{-1mm}{
			\input{figures/covering_grid_ratios.pgf}
		}
		\label{fig:gridratios}
	}
	\caption{%
		Application of geometric grid cover to linearized quadrotor.
		(a) Illustration of geometric grid partition.
		(b) Empirical upper bound on covering number. %
		(c) Suboptimality ratios for corner cells in empirical cover.
		Discussion in \Cref{sec:grid}.
	}
\end{figure}

\paragraph{Empirical upper bound on $\coveringnum{\alpha}{\Envs}$.}
In this experiment, we construct an $\alpha$-suboptimal cover
$\coveringset$
using geometric grids,
such that each $K \in \coveringset$
is $\alpha$-suboptimal for a full grid cell.
For each cell $\bSigma(j)$, we attempt GCC synthesis.
If it succeeds,
we check if $\bSigma(j) \subseteq \suboptneighb{\alpha}{K(j)}$.
If not, we increment the grid pitch $k$ and try again.
Termination is guaranteed by continuity.
We show results for %
the linearized quadrotor %
with $\alpha = 2$ in
\Cref{fig:quadrotor_covering}.
The data follow roughly logarithmic growth,
as indicated by the linear least-squares best-fit curve in black.
Small values of $\breadth$ are excluded from the fit (indicated by grey points),
as we do not expect the asymptotic growth pattern to appear yet.

These results do not rule out
the $\log(\theta)^d$ growth suggested by
the geometric grid construction.
Testing larger values of $\breadth$
is computationally difficult because the number of grid cells becomes huge
and the GCC Riccati equation becomes numerically unstable for very small $\Sigma$.

\paragraph{Efficiency of geometric grid partition.}
\label{sec:are-geometric-efficient}

Given an $\alpha$-suboptimal geometric grid cover,
we examine a measurable quantity that may reflect the ``efficiency'' of the cover.
Intuitively,
in a good cover we expect 
the suboptimality ratio of each controller $K(j)$ relative to its grid cell $\bSigma(j)$
to be close to $\alpha$.
If it close to $\alpha$ for some cells but significantly less than $\alpha$ for others,
then the grid pitch around the latter cells is finer than necessary.
We visualize results for this computation on the linearized quadrotor
with $\theta = 10,\ k = 4$
in \Cref{fig:gridratios}
---
only the corners of the $4 \times 4 \times 4 \times 4$ grid are shown.
The suboptimality ratio is close to $\alpha = 2$
for cells with low control authority (near $\Sigma = \recip \theta I$),
but drops to around $1.4$
for cells with high control authority (near $\Sigma = I$).
The difference suggests that the geometric grid cover could be more efficient in the high-authority regime.

\paragraph{Efficiency of GCC synthesis.}
One possible source of conservativeness is that
\lemmaref{lem:petersen-gcc} applies to the affine image of
a $m \times n$-dimensional matrix norm ball,
but we only require guaranteed cost on
a $d$-dimensional affine subspace of diagonal matrices.
In other words, we ask GCC synthesis to ensure $\alpha$-suboptimality on systems that are not actually part of $\Envs$.
If this is negatively affecting the result,
then we should observe that the worst-case cost of $K(j)$ on $\bSigma(j)$
is less than the trace of the solution $P$ for the GCC Riccati equation \eqref{eq:petersen-riccati}.
The worst-case cost always occurs at the minimal $\Sigma \in \bSigma(j)$ by \lemmaref{lem:lancaster-ARE-domination};
we evaluate it with \eqref{eq:cost-suboptimal}.
For the quadrotor,
a mismatch sometimes occurs for smaller values of $\theta$,
but it does not occur for the large values of $\theta$.

\subsection{Suboptimal neighborhood visualizations}
\label{sec:matrix-lb-ideas}

We now present intuition-building experiments towards a covering number lower bound
for matrix \DDFproblems.
A lower bound requires
a class of \DDFproblem\
that can be instantiated for any dimensionality $d$.
Two choices come to mind:
\emph{minimum coupling},
where $A = \textstyle I$,
and \emph{maximum coupling},
where $A = \textstyle \recip n \one$.
Note that for minimum coupling,
an $\alpha$-suboptimal policy
is not necessarily $\alpha$-suboptimal on each scalar subsystem---%
if it were, the lower bound $\log (\theta)^d$ would trivially follow
from the results in \Cref{sec:results-scalar}.

\label{sec:2d-visualization}

\begin{figure}[t]
	\(
		\hspace{7mm}
		\overbrace{\hspace{65mm}}^{A \mathop{=} I}
		\hspace{10mm}
		\overbrace{\hspace{65mm}}^{A \mathop{=} \recip n \one}
	\)
	\adjustbox{clip,trim=4.5mm 1mm 0 4.5mm}{
		\input{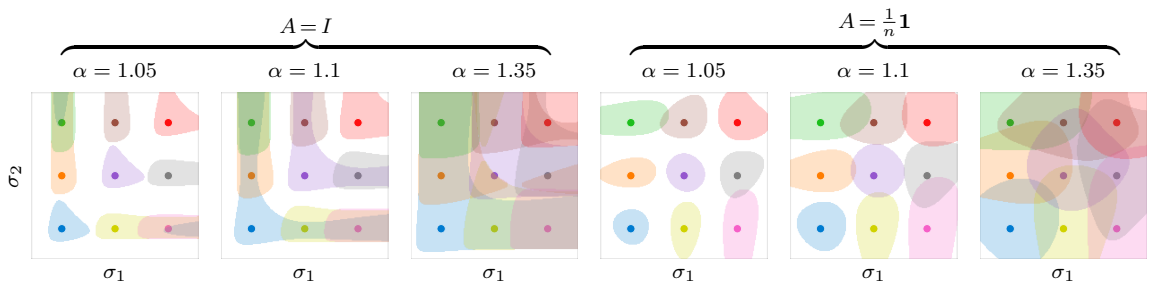}
	}
	\vspace{-2mm}
	\caption{
		$\alpha$-suboptimal neighborhoods
		for geometric grid partition in 2D system.
		\emph{Left:}
		minimum coupling; $A = I$.
		\emph{Right:}
		maximum coupling; $A = \recip n \one$.
		\emph{Columns:} varying suboptimality threshold $\alpha$.
		All axes are logarithmic. %
		Colors have no meaning.
		Discussion in \Cref{sec:2d-visualization}.
	}
	\label{fig:neighborhoods-2x2}
\end{figure}

We show approximate suboptimal neighborhoods
for a two-dimensional system in \Cref{fig:neighborhoods-2x2}.
We select a geometric grid of $\Sigma$ values (indicated by the circular markers)
and synthesize their LQR-optimal controllers.
Then, we evaluate the suboptimality ratio of each controller on a finer grid of $\Sigma$ values
to get approximate neighborhoods, indicated by the semi-transparent regions.
We repeat this experiment
with three values of $\alpha$
for both choices of $A$.

Interestingly, the neighborhoods for $A=I$ are not always connected.
In the plot for ${\alpha = 1.05}$ (far left),
the neighborhood for the minimal $\Sigma$
has another component that overlaps other neighborhoods
to its top and right.
If we increase to $\alpha = 1.1$,
the components join
into an ``L''-shaped region.
In contrast, the neighborhoods for $A = \recip n \one$ seem more well-behaved.
For both choices of $A$,
the neighborhoods are of comparable size.

To verify that this behavior %
is not an artifact of the two-dimensional case only,
we repeat the experiment in three dimensions.
\Cref{fig:neighborhoods-3x3} shows neighborhoods
of one controller $K = \Kopt{(2/\theta) I}$
for $\alpha$ ranging from $1.04$ to $1.2$.
As $\alpha$ grows, $\suboptneighb{\alpha}{K}$ shows similar topological phases as the $2$D case.
In the simply-connected phase (large $\alpha$),
the neighborhood appears to include
any $\Sigma$ where at least one $\sigma_i$ is sufficiently small.
If this property holds in higher dimensions,
then it would be possible
to construct a cover using only controllers of uniform gain in all dimensions
for large $\alpha$.

\begin{figure}[t]
	\hfill
	\foreach \alpha in {1.0367,1.0526,1.0852,1.1137,1.2000}
	{
		\includegraphics[width=0.17\textwidth,trim={60px 130px 80px 160px},clip]{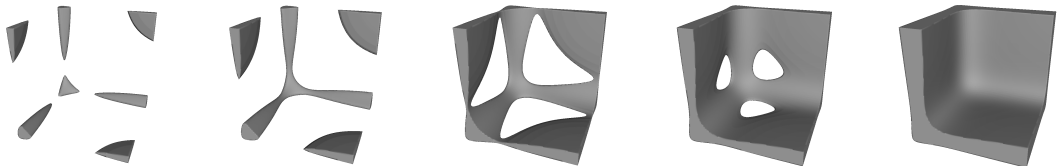}
		\hfill
	}
	\caption{%
		$\alpha$-suboptimal neighborhoods for the three-dimensional decomposed dynamics system with
		minimal coupling
		($A = U = V^\top = I_{3 \times 3}$)
		and breadth $\breadth = 100$.
		Neighborhoods shown for
		$\alpha$ ranging from $1.04$ to $1.2$
		with a fixed
		controller.
	}
	\label{fig:neighborhoods-3x3}
\end{figure}

\section{Related work}
\label{sec:related}

Suboptimal coverings are closely related to several topics in control theory.
Robust control synthesis under parametric uncertainty
\citep{dullerud-robust}
can be interpreted as seeking a policy that performs well on all of $\Envs$
without observing the particular ${\env \in \Envs}$.
Most problem statements in robust synthesis admit problem instances with no solution; the goal is to find a robust policy \emph{if} one exists.
Adaptive control is also concerned with sets of control tasks,
with the added complication that $\env$ is not known to the policy.
Adaptive policies
of the self-tuning type
synthesize a \singleenv\ policy after estimating $\env$,
but this relies on the assumption that control synthesis can be computed quickly
\citep{astrom-adaptive}.

Adaptive and gain-scheduled multi-model methods use a precomputed set of policies instead
\citep{murray-smith-multimodel},
but researchers have focused more on the switching rule than the policy set.
For example,
\citet{fu-adaptive-switching,stilwell-interpolation,yoon-gain-scheduling-minimax}
non-constructively assert the existence of a finite cover
by continuity and compactness arguments.
To address the need for small covers,
\citet{anderson-coverings,mcnichols-selecting-discrete,tan-selecting,fekri-issues,du-multimodel-integrated}
propose constructive algorithms, %
sometimes with arguments for minimality,
but without size bounds on the cover.
\citet{jalali-multimodel-vinnicombe}
show an upper bound in terms of frequency-domain properties of $\Envs$,
as opposed to state-space parameters like mass and geometry.
The most closely related work to ours is from \citet{fu-minimum-switching},
who shows a tight bound of $2^n$
for the stability covering number
of a relatively broad $\Envs$. %
This result is complementary to ours:
suboptimality is a stronger criterion than stability,
but our class of $\Envs$ is more restrictive.
We are not aware of prior work that bounds covering numbers in a setup based on local suboptimality,
as opposed to a single global performance measure.

\Multienv\ control is also a popular topic in deep learning research,
where it is often motivated by ideas of lifelong skill acquisition in robotics.
Domain randomization methods
follow the spirit of robust control
\citep{peng-dynamics-randomization-corr17},
but usually optimize for the average case instead of a worst-case guarantee.
Many methods where the policy observes $\env$
use architectural constructs that can only be applied to finite \envword\ sets
\citep{yang-multitask,parisotto-actormimic,devin-modular}.
A common approach for infinite \envword\ spaces is to treat $\env$ as a vector input alongside the system state.
\citet{yu-up-osi-rss17,chen-hardware-conditioned}
use this approach for dynamics parameters;
\citet{schaul-uvfa} use it for navigation goals.
There is evidence that policy class influences these methods:
in a recent benchmark %
\citep{yu-metaworld},
the concatenated-input architecture
trails the finite-task architecture.
Other investigations into the difficulty of learning policies for \multienv\ control include
methods to condition the \multienv\ optimization landscape
\citep{yu-pcgrad}
or balance disparate cost ranges
\citep{hasselt-popart}.

\section{Conclusion and future work}
In this paper, we introduced and motivated the
$\alpha$-suboptimal covering number
to quantify infinite \envword\ spaces
for \multienv\ control problems.
We defined a particular class of \multienv\ linear-quadratic regulator problems
amenable to analysis of the $\alpha$-suboptimal covering number,
and showed logarithmic dependency on the problem ``breadth'' parameter $\breadth$ in the scalar case.
Towards analogous results for the matrix case,
we presented empirical studies intended to shed light on possible proof techniques.
For the upper bound, we considered a natural covering construction that would preserve logarithmic dependence on $\breadth$
but give exponential dependence on dimensionality.
Experiments did not rule out its validity.
For the lower bound, we visualized suboptimal neighborhoods for two possible system classes
and observed interesting topological behavior for the minimal-coupling class.

After extending our current results to the matrix case,
in future work the analysis can be applied to other classes of \multienv\ LQR problems
including variations in $A, Q, R$,
discrete time, and stochastic dynamics.
It will be interesting to see if there are major differences between LQR variants.
We also hope that suboptimal covers and covering numbers will be a useful tool
for analyzing how the size of the \envword\ space
affects the required expressiveness of
function classes used in practice as \multienv\ policies,
such as neural networks.

\newpage
\bibliography{math, lqr, deep, rl}

\begin{thebibliography}{29}
\providecommand{\natexlab}[1]{#1}
\providecommand{\url}[1]{\texttt{#1}}
\expandafter\ifx\csname urlstyle\endcsname\relax
  \providecommand{\doi}[1]{doi: #1}\else
  \providecommand{\doi}{doi: \begingroup \urlstyle{rm}\Url}\fi

\bibitem[Anderson et~al.(2000)Anderson, Brinsmead, De~Bruyne, Hespanha,
  Liberzon, and Morse]{anderson-coverings}
Brian D.~O. Anderson, Thomas~S. Brinsmead, Franky De~Bruyne, Joao Hespanha,
  Daniel Liberzon, and A.~Stephen Morse.
\newblock Multiple model adaptive control. part 1: Finite controller coverings.
\newblock \emph{International Journal of Robust and Nonlinear Control},
  10\penalty0 (11-12):\penalty0 909--929, 2000.

\bibitem[{\AA}str{\"o}m and Wittenmark(2013)]{astrom-adaptive}
Karl~J {\AA}str{\"o}m and Bj{\"o}rn Wittenmark.
\newblock \emph{Adaptive Control}.
\newblock Courier Corporation, 2013.

\bibitem[Boyd and Vandenberghe(2004)]{boyd-convex}
Stephen Boyd and Lieven Vandenberghe.
\newblock \emph{Convex Optimization}.
\newblock Cambridge University Press, 2004.

\bibitem[Bu et~al.(2019)Bu, Mesbahi, and Mesbahi]{bu-topological-gains}
Jingjing Bu, Afshin Mesbahi, and Mehran Mesbahi.
\newblock On topological and metrical properties of stabilizing feedback gains:
  the {MIMO} case.
\newblock \emph{CoRR}, abs/1904.02737, 2019.

\bibitem[Chen et~al.(2018)Chen, Murali, and Gupta]{chen-hardware-conditioned}
Tao Chen, Adithyavairavan Murali, and Abhinav Gupta.
\newblock Hardware conditioned policies for multi-robot transfer learning.
\newblock In \emph{NeurIPS}, pages 9355--9366, 2018.

\bibitem[Devin et~al.(2017)Devin, Gupta, Darrell, Abbeel, and
  Levine]{devin-modular}
Coline Devin, Abhishek Gupta, Trevor Darrell, Pieter Abbeel, and Sergey Levine.
\newblock Learning modular neural network policies for multi-task and
  multi-robot transfer.
\newblock In \emph{IEEE International Conference on Robotics and Automation
  (ICRA)}, pages 2169--2176, 2017.

\bibitem[Du et~al.(2012)Du, Song, and Li]{du-multimodel-integrated}
Jingjing Du, Chunyue Song, and Ping Li.
\newblock Multimodel control of nonlinear systems: An integrated design
  procedure based on gap metric and {$H_\infty$} loop shaping.
\newblock \emph{Industrial \& Engineering Chemistry Research}, 51\penalty0
  (9):\penalty0 3722--3731, 2012.

\bibitem[Dullerud and Paganini(2000)]{dullerud-robust}
Geir~E. Dullerud and Fernando Paganini.
\newblock \emph{A Course in Robust Control Theory: A Convex Approach}.
\newblock Springer-Verlag New York, 2000.

\bibitem[Fazel et~al.(2018)Fazel, Ge, Kakade, and Mesbahi]{fazel-global-lqr}
Maryam Fazel, Rong Ge, Sham~M. Kakade, and Mehran Mesbahi.
\newblock Global convergence of policy gradient methods for linearized control
  problems.
\newblock \emph{CoRR}, abs/1801.05039, 2018.

\bibitem[Fekri et~al.(2006)Fekri, Athans, and Pascoal]{fekri-issues}
Sajjad Fekri, Michael Athans, and Antonio Pascoal.
\newblock Issues, progress and new results in robust adaptive control.
\newblock \emph{International Journal of Adaptive Control and Signal
  Processing}, 20\penalty0 (10):\penalty0 519--579, 2006.

\bibitem[Fu(1996)]{fu-minimum-switching}
Minyue Fu.
\newblock Minimum switching control for adaptive tracking.
\newblock In \emph{Proceedings of 35th IEEE Conference on Decision and
  Control}, volume~4, pages 3749--3754, 1996.

\bibitem[Fu and Barmish(1986)]{fu-adaptive-switching}
Minyue Fu and B.~Ross Barmish.
\newblock Adaptive stabilization of linear systems via switching control.
\newblock \emph{IEEE Transactions on Automatic Control}, 31\penalty0
  (12):\penalty0 1097--1103, 1986.

\bibitem[Jalali and Golmohammad(2012)]{jalali-multimodel-vinnicombe}
Ali~Akbar Jalali and Hassan Golmohammad.
\newblock An optimal multiple-model strategy to design a controller for
  nonlinear processes: A boiler-turbine unit.
\newblock \emph{Computers \& Chemical Engineering}, 46:\penalty0 48--58, 2012.

\bibitem[Kalman(1960)]{kalman-contributions}
R.~E. Kalman.
\newblock Contributions to the theory of optimal control.
\newblock \emph{Bolet\'in de la Sociedad Matem\'atica Mexicana}, 5:\penalty0
  102--199, 1960.

\bibitem[Lancaster and Rodman(1995)]{lancaster-AREs}
Peter Lancaster and Leiba Rodman.
\newblock \emph{Algebraic {Riccati} Equations}.
\newblock Clarendon Press, 1995.

\bibitem[McNichols and Fadali(2003)]{mcnichols-selecting-discrete}
Kenneth~H. McNichols and M.~Sami Fadali.
\newblock Selecting operating points for discrete-time gain scheduling.
\newblock \emph{Computers \& Electrical Engineering}, 29\penalty0 (2):\penalty0
  289--301, 2003.

\bibitem[Murray-Smith and Johansen(1997)]{murray-smith-multimodel}
Roderick Murray-Smith and Tor~Arne Johansen.
\newblock \emph{Multiple Model Approaches to Modelling and Control}.
\newblock Taylor and Francis, London, 1997.

\bibitem[Parisotto et~al.(2016)Parisotto, Ba, and
  Salakhutdinov]{parisotto-actormimic}
Emilio Parisotto, Lei~Jimmy Ba, and Ruslan Salakhutdinov.
\newblock Actor-mimic: Deep multitask and transfer reinforcement learning.
\newblock In \emph{International Conference on Learning Representations
  ({ICLR})}, 2016.

\bibitem[Peng et~al.(2017)Peng, Andrychowicz, Zaremba, and
  Abbeel]{peng-dynamics-randomization-corr17}
Xue~Bin Peng, Marcin Andrychowicz, Wojciech Zaremba, and Pieter Abbeel.
\newblock Sim-to-real transfer of robotic control with dynamics randomization.
\newblock \emph{CoRR}, abs/1710.06537, 2017.

\bibitem[Petersen and McFarlane(1994)]{petersen-guaranteed-cost}
Ian~R. Petersen and Duncan~C. McFarlane.
\newblock Optimal guaranteed cost control and filtering for uncertain linear
  systems.
\newblock \emph{IEEE Transactions on Automatic Control}, 39\penalty0
  (9):\penalty0 1971--1977, 1994.

\bibitem[Schaul et~al.(2015)Schaul, Horgan, Gregor, and Silver]{schaul-uvfa}
Tom Schaul, Daniel Horgan, Karol Gregor, and David Silver.
\newblock Universal value function approximators.
\newblock volume~37 of \emph{Proceedings of Machine Learning Research}, pages
  1312--1320, Lille, France, Jul 2015.

\bibitem[Stilwell and Rugh(1999)]{stilwell-interpolation}
Daniel~J. Stilwell and Wilson~J. Rugh.
\newblock Interpolation of observer state feedback controllers for gain
  scheduling.
\newblock \emph{IEEE Transactions on Automatic Control}, 44\penalty0
  (6):\penalty0 1225--1229, 1999.

\bibitem[Tan et~al.(2004)Tan, Marquez, and Chen]{tan-selecting}
Wen Tan, Horacio~J. Marquez, and Tongwen Chen.
\newblock Operating point selection in multimodel controller design.
\newblock In \emph{Proceedings of the 2004 American Control Conference},
  volume~4, pages 3652--3657, 2004.

\bibitem[van Hasselt et~al.(2016)van Hasselt, Guez, Hessel, Mnih, and
  Silver]{hasselt-popart}
Hado~P van Hasselt, Arthur Guez, Matteo Hessel, Volodymyr Mnih, and David
  Silver.
\newblock Learning values across many orders of magnitude.
\newblock In \emph{Advances in Neural Information Processing Systems}, pages
  4287--4295, 2016.

\bibitem[Yang et~al.(2017)Yang, Merrick, Abbass, and Jin]{yang-multitask}
Zhaoyang Yang, Kathryn~E Merrick, Hussein~A Abbass, and Lianwen Jin.
\newblock Multi-task deep reinforcement learning for continuous action control.
\newblock In \emph{IJCAI}, pages 3301--3307, 2017.

\bibitem[Yoon et~al.(2007)Yoon, Ugrinovskii, and
  Pszczel]{yoon-gain-scheduling-minimax}
Myung-Gon Yoon, Valery~A. Ugrinovskii, and Marek Pszczel.
\newblock Gain-scheduling of minimax optimal state-feedback controllers for
  uncertain {LPV} systems.
\newblock \emph{IEEE Transactions on Automatic Control}, 52\penalty0
  (2):\penalty0 311--317, 2007.

\bibitem[Yu et~al.(2019)Yu, Quillen, He, Julian, Hausman, Finn, and
  Levine]{yu-metaworld}
Tianhe Yu, Deirdre Quillen, Zhanpeng He, Ryan Julian, Karol Hausman, Chelsea
  Finn, and Sergey Levine.
\newblock Meta-world: {A} benchmark and evaluation for multi-task and meta
  reinforcement learning.
\newblock In \emph{CoRL}, volume 100 of \emph{Proceedings of Machine Learning
  Research}, pages 1094--1100. {PMLR}, 2019.

\bibitem[Yu et~al.(2020)Yu, Kumar, Gupta, Levine, Hausman, and Finn]{yu-pcgrad}
Tianhe Yu, Saurabh Kumar, Abhishek Gupta, Sergey Levine, Karol Hausman, and
  Chelsea Finn.
\newblock Gradient surgery for multi-task learning.
\newblock \emph{CoRR}, abs/2001.06782, 2020.

\bibitem[Yu et~al.(2017)Yu, Tan, Liu, and Turk]{yu-up-osi-rss17}
Wenhao Yu, Jie Tan, C.~Karen Liu, and Greg Turk.
\newblock Preparing for the unknown: Learning a universal policy with online
  system identification.
\newblock In \emph{Robotics: Science and Systems (RSS)}, 2017.

\end{thebibliography}

\ifextended{
   \newpage
   \appendix
   \section{Extended theoretical results}
   \label{sec:appendix-theory}

This appendix contains the proof for \lemmaref{lem:subopt-convex}.
We first present some supporting material.

\subsection{Supporting material}
We begin by recalling the definition and properties of quasiconvex functions on $\R$.
We then recall some basic facts about scalar LQR.

\begin{definition}
	A function $f : \cD \mapsto \R\ $ on the convex domain $\ \cD \subseteq \R^n$
	is \emph{quasiconvex} if its sublevel sets
	\(
		\cD_\alpha = \{ x \in \cD : f(x) \leq \alpha \}
	\)
	are convex for all $\alpha \in \R$.
\end{definition}

\begin{lemma}[\citet{boyd-convex}, \S 3.4]
\label{lem:quasiconvex}
	The following facts hold for quasiconvex functions on a convex $\cD \subseteq \R$:
	\begin{enumerate}[label=(\alph*)]
		\item \label{alternatives}
			If $f: \R \mapsto \R$ is continuous,
			then $f$ is quasiconvex
			if and only if at least one of the following conditions holds on $\cD$:
			\begin{enumerate}[label=\arabic*.]
				\item $f$ is nondecreasing.
				\item $f$ is nonincreasing.
				\item \label{alternative-incdec}
					There exists $c \in \cD$ such that
					for all $t \in \cD$,
					if $t < c$ then $f$ is nonincreasing,
					and if $t \geq c$ then $f$ is nondecreasing.
			\end{enumerate}
		\item \label{secondorder}
			If $f : \R \mapsto \R$ is twice differentiable
			and
			$\partial^2 f / \partial x^2 > 0$
			for all $x \in \cD$ where
			$\partial f / \partial x = 0$,
			then $f$ is quasiconvex on $\cD$.
		\item \label{quotient}
			If $f(x) = p(x) / q(x)$,
			where $p: \R \mapsto \R$ is convex with $p(x) \geq 0$ on $\cD$
			and $q: \R \mapsto \R$ is concave with $q(x) > 0$ on $\cD$,
			then $f$ is quasiconvex on $\cD$.
	\end{enumerate}
\end{lemma}

\noindent The following facts about scalar LQR problems can be derived from
the LQR Riccati equation and some calculus (not shown).
\begin{lemma}
\label{lem:scalar-lqr-facts}
For the scalar LQR problem with $a > 0, b > 0$ and $q = r = 1$,
the optimal linear controller $\kopt{a,b}$ is given by the closed-form expression
	\[
		k^\star_{a,b}
		=
		\min_{k \in \R} J_{a,b}(k)
		=
		\min_{k \in \R} \frac{1 + k^2}{-2(a + bk)}
		=
		-\frac{a + \sqrt{a^2 + b^2}}{b}.
	\]
For fixed $a$, the map from $b$ to $k^\star_{a,b}$
is continuous and strictly increasing on the domain $b \in (0, \infty)$
and has the range $(-\infty, -1)$.
For any $k \in (-\infty, -1)$,
there exists a unique $b_k \in (-\infty, -1)$
for which $k = k^\star_{a,b_k}$, given by
\[
	b_k = \frac{2ak}{1 - k^2}.
\]
\end{lemma}

\newcommand{\bkdomain}{\cD}
\newcommand{\bdomain}{\cD^b}
\newcommand{\kdomain}{\cD^k}

\subsection{Proof of \lemmaref{lem:subopt-convex}}

We first recall the statement of the lemma.
\quasi*

\noindent Instead of a monolithic proof,
we present supporting material in \lemmaref{lem:ratio-quasi,lem:cost-quasi}.
We then show the main result in \lemmaref{lem:increasing-main},
which considers $\alpha$-suboptimal neighborhoods on all of $\R$
instead of restricted to $\Bset$.
\lemmaref{lem:subopt-convex} will follow as a corollary.

We proceed with more setup.
Recall that the scalar \DDFproblem\ is defined by
$\Aset = \{a\}$ and $\Bset = [\recip \theta, 1]$, where $a > 0$.
For this section,
let
\[
	\bkdomain
	= \{ (b, k) \in (0, \infty) \times \R : a + bk < 0 \}
\]
(note that $J_b(k) < \infty \iff a + bk < 0$).
Denote its projections by
$\bdomain(k) = \{ b: (b, k) \in \bkdomain \}$
and
$\kdomain(b) = \{ k: (b, k) \in \bkdomain \}$.
We compute the suboptimality ratio $r : \bkdomain \mapsto \R$ by
\[
	r(b, k) =
	\frac{J_b(k)}{J^\star_b}
	=
	\nofrac{
		\frac{
			1 + k^2
		}{
			-2(a + bk)
		}
	}{
		\frac{
			a + \sqrt{a^2 + b^2}
		}{
			b^2
		}
	}
	=
	\frac{
		(1 + k^2) b^2
	}{
		-2(a + bk)(a + \sqrt{a^2 + b^2})
	}.
\]
We denote its sublevel sets with respect to $b$ for fixed $k$ by
\[
	\bdomain_\alpha(k) = \{ b \in \bdomain(k) : r(b, k) \leq \alpha \}.
\]
\begin{lemma}
\label{lem:ratio-quasi}
	For fixed $k < 0$, the ratio $r(b, k)$ is quasiconvex on $\bdomain_k$,
	and there is at most one $b \in \bdomain_k$ at which $\partial r / \partial b = 0$.
\end{lemma}
\begin{proof}
	By inspection, $r(b, k)$ is smooth on $\bdomain$.
	We now show that the second-order condition of \lemmaref{lem:quasiconvex}\ref{secondorder} holds.
	To solve $\partial r / \partial b = 0$ for $b$,
we multiply $\partial r/ \partial b$ (not shown due to length)
by the strictly positive factor
\[
    \frac{2 \left(a + b k\right)^{2} \left(a + \sqrt{a^{2} + b^{2}}\right)^{2} \sqrt{a^{2} + b^{2}}}{a b \left(k^{2} + 1\right)}
\]
and set the result equal to zero to get the equation
\[
    2 a^{2} + a b k + b^{2} = \left(- 2 a - b k\right) \sqrt{a^{2} + b^{2}}.
\]
Squaring both sides (which may introduce spurious solutions)
and collecting terms yields the equation
\(
    - 2 a k - b k^{2} + b = 0,
\)
with the solution $b = {}$ $\frac{2 a k}{1 - k^{2}}$.
This is the expression for $b_k$ from \lemmaref{lem:scalar-lqr-facts}.
Note that it is only positive for $k < -1$.
If $k \in [-1, 0)$, then there are no stationary points in $\bdomain_k$.
Otherwise, substitution into $\partial r / \partial b$ confirms that this solution is not spurious,
so it is the only stationary point of $r$ with respect to $b$.
We now must check the second-order condition for $k < -1$.
Evaluating $\partial^2 r / \partial b^2$ (not shown due to length)
and multiplying by the strictly positive factor
\[
    - \frac{\left(a + \sqrt{a^{2} + b^{2}}\right) \left(2 a + 2 b k\right)}{k^{2} + 1},
\]
we have
\begin{dmath*}
    \sign \left( \frac{\partial^2 r}{\partial b^2} \right)
    =
    \sign \left( \frac{b^{4}}{\left(a + \sqrt{a^{2} + b^{2}}\right) \left(a^{2} + b^{2}\right)^{\frac{3}{2}}} + \frac{2 b^{4}}{\left(a + \sqrt{a^{2} + b^{2}}\right)^{2} \left(a^{2} + b^{2}\right)} + \frac{2 b^{3} k}{\left(a + b k\right) \left(a + \sqrt{a^{2} + b^{2}}\right) \sqrt{a^{2} + b^{2}}} + \frac{2 b^{2} k^{2}}{\left(a + b k\right)^{2}} - \frac{5 b^{2}}{\left(a + \sqrt{a^{2} + b^{2}}\right) \sqrt{a^{2} + b^{2}}} - \frac{4 b k}{a + b k} + 2 \right).
\end{dmath*}
Evaluating at the stationary point $b_k$, this reduces to
\begin{dmath*}
    \left. \sign \left( \frac{\partial^2 r}{\partial b^2} \right) \right|_{b_k, k}
    =
    \sign \left( \frac{2 k^{2} \left(k - 1\right)^{2} \left(k + 1\right)^{2}}{\left(k^{2} + 1\right)^{3}} \right).
\end{dmath*}
Recalling that $k < -1$, the sign is positive.

	The conclusion follows from \lemmaref{lem:quasiconvex}\ref{secondorder}.
\end{proof}

\begin{lemma}
\label{lem:cost-quasi}
	For fixed $b$, the cost $J_b(k)$ is quasiconvex on $\kdomain(b)$.
	Also, $J_b(k)$ is not monotonic,
	so case \ref{alternative-incdec} of \lemmaref{lem:quasiconvex}\ref{alternatives} applies.
\end{lemma}
\begin{proof}
	We have
	\[
		J_b(k)
		=
		\frac{
			1 + k^2
		}{
			-2(a + bk)
		}.
	\]
	The numerator is nonnegative and convex on $k \in \R$.
	The denominator is linear (hence concave) and positive on $\kdomain(b)$.
	Quasiconvexity follows from \lemmaref{lem:quasiconvex}\ref{quotient}.
	Nonmonotonicity follows from the fact that $J_b(k)$ is smooth on $\kdomain(b)$
	and has a unique optimum at $\kopt{b}$, which is not on the boundary of $\kdomain(b)$.
\end{proof}

\noindent We now combine these into the main result.

\begin{lemma}
\label{lem:increasing-main}
	For a scalar \DDFproblem,
	if $\alpha > 1$ and $k < -1$,
	then $\bdomain_\alpha(k)$ is either:
	a bounded closed interval $[b_1, b_2]$,
		with $b_1$ and $b_2$ increasing in $k$, or
	a half-bounded closed interval $[b_1, \infty)$,
		with $b_1$ increasing in $k$.
\end{lemma}

\begin{proof}
	By \lemmaref{lem:ratio-quasi}, due to quasiconvexity $\bdomain_\alpha$ is convex.
	The only convex sets on $\R$ are the empty set and all types of intervals:
	open, closed, and half-open.
	We know $\bdomain_\alpha$ is not empty because it contains $b_k$.
	We can further assert that
	$\bdomain_\alpha$ has a closed lower bound
	because $\lim_{b \to (-a/k)} r(b, k) = \infty$
	(see \citet[\S A.3.3]{boyd-convex} for details).
	However, the upper bound may be closed or infinite.
	We handle the two cases separately.

	\paragraph{Bounded case.}
	Fix $k_0 < -1$.
	Suppose $\bdomain_\alpha(k_0) = [b_1, b_2]$ for $0 < b_1 < b_2 < \infty$.
	By the implicit function theorem (IFT), at any $(b_0, k_0)$ satisfying $r(b_0, k_0) = \alpha$,
	if $\partial r / \partial b|_{b_0, k_0} \neq 0$
	then there exists an open neighborhood around $(b_0, k_0)$
	for which the solution to $r(b, k) = \alpha$ can be expressed as
	$(g(k), k)$, where $g$ is a continuous function of $k$ and
	\[
		\left. \frac{\partial g(k)}{\partial k} \right|_{k_0} =
		\left. - \left( \frac{\partial r}{\partial b} \right)^{-1}
		\frac{\partial r}{\partial k} \right|_{b_0, k_0}.
	\]
	By the continuity and quasiconvexity of $r$,
	and the fact that $\partial r / \partial b = 0$ only at $b_k$
	(\lemmaref{lem:ratio-quasi})
	we know that $r(b_1, k_0) = r(b_2, k_0) = \alpha$
	and
	\[
		\left. \frac{\partial r}{\partial b} \right|_{b_1, k_0} < 0
		\quad \text{and} \quad
		\left. \frac{\partial r}{\partial b} \right|_{b_2, k_0} > 0.
	\]
	By \lemmaref{lem:scalar-lqr-facts},
	since $k < -1$
	there exists $b_k > 0$ satisfying
	$k = k^\star_{b_k}$.
	Since $r(b_k, k) = 1$ and $\alpha > 1$, we know $b_{k_0} \in (b_1, b_2)$.
	Again by \lemmaref{lem:scalar-lqr-facts},
	the map from $b$ to $k^\star_b$ is increasing in $b$.
	Therefore,
	\(
		{k^\star_{b_1} < k_0 < k^\star_{b_2}}.
	\)
	By the quasiconvexity and nonmonotonicity of $J_b(k)$ from \lemmaref{lem:cost-quasi},
	via \lemmaref{lem:quasiconvex}\ref{alternatives}
	we have
	\[
		\left. \frac{\partial r}{\partial k} \right|_{b_1, k_0} \geq 0
		\quad \text{and} \quad
		\left. \frac{\partial r}{\partial k} \right|_{b_2, k_0} \leq 0.
	\]
	Therefore, the functions $g_1, g_2$ satisfying the conclusion of the IFT
	in the neighborhoods around $(b_1, k_0)$ and $(b_2, k_0)$ respectively
	also satisfy
	\[
		\left. \frac{ \partial g_1(k)}{ \partial k } \right|_{b_1, k_0} \geq 0
		\quad \text{and} \quad
		\left. \frac{ \partial g_2(k)}{ \partial k } \right|_{b_2, k_0} \geq 0.
	\]
	Therefore, $b_1$ and $b_2$ are locally nondecreasing in $k$.

	\paragraph{Unbounded case.}
	Suppose $\bdomain_\alpha(k) = [b_1, \infty)$ for $b_1 < \infty$.
	By the same IFT argument as in the bounded case, $b_1$ is increasing in $k$.
	By the quasiconvexity of $r$ in $b$,
	the value of $r$ is increasing for $b  > b_k$,
	but the definition of $\bdomain_\alpha(k)$ implies that
	$r(b, k) \leq \alpha$ for all $b > b_k$.
	Therefore, $\lim_{b \to \infty} r(b, k)$ exists
	and is bounded by $\alpha$.
	In particular,
	\begin{equation*}
		\begin{split}
			\lim_{b \to \infty} r(b, k)
			&=
			\lim_{b \to \infty}
			\frac{
				(1 + k^2) b^2
			}{
				-2(a + bk)(a + \sqrt{a^2 + b^2})
			}
			\\ &=
			\lim_{b \to \infty}
			\frac{
				(1 + k^2) b^2 / b^2
			}{
				-2(a + bk)(a + \sqrt{a^2 + b^2}) / b^2
			}
			\\ &=
			-
			\frac{
				1 + k^2
			}{
				2k
			}.
		\end{split}
	\end{equation*}
	Taking the derivative shows that this value is decreasing in $k$ for $k < 0$.
	Therefore, if $k < k' < 0$ then
	\[
		\lim_{b \to \infty} r(b, k') \leq \lim_{b \to \infty} r(b, k) \leq \alpha.
	\]
	The property that $r(b, k')$ is increasing in $b$ for $b > b_k$
	further ensures that
	$r(b, k') \leq \alpha$ for all $b > b_k$.
	Therefore, $\bdomain_\alpha(k')$ is also unbounded.
\end{proof}

\noindent For completeness, we prove \lemmaref{lem:subopt-convex}.

\quad

\begin{proof} (of \lemmaref{lem:subopt-convex}).
	By \lemmaref{lem:increasing-main},
	$\bdomain_\alpha(k)$ is either
	a bounded closed interval $[b_1, b_2]$,
		with $b_1$ and $b_2$ increasing in $k$, or
	a half-bounded closed interval $[b_1, \infty)$,
		with $b_1$ increasing in $k$.
	Recall that $\suboptneighb{\alpha}{k} = \bdomain_\alpha(k) \cap \Bset$
	with $\Bset = [\recip \theta, 1]$.
	Therefore, the half-bounded case can be reduced to the bounded case with $b_2 = 1$.
	The intersection can be expressed as
	\[
		\suboptneighb{\alpha}{k} = [\max\{b_1, \textstyle \recip \theta\}, \min\{b_2, 1\}],
	\]
	where the interval $[a, b]$ is defined as the empty set if $a > b$.
	Taking the maximum or minimum of a nonstrict monotonic function and a constant
	preserves the monotonicity, so we are done.
\end{proof}

}

\end{document}